\documentclass[12pt,reqno]{amsart} 

\usepackage{amssymb,amsthm,bbm,enumitem,esint,url} 
\usepackage{aliascnt}
\usepackage{doi}
\usepackage{pgfplots}

\usepackage[margin=1.25in]{geometry}

\usepackage{color}

\usepackage{psfrag}

\newcommand{\arxiv}[1]{%
 \href{https://arxiv.org/pdf/#1.pdf}{ArXiv:#1}}

\newtheorem*{theorem*}{Theorem}
\newtheorem{thmx}{Theorem}

\newtheorem{theorem}{Theorem}[section]

\newaliascnt{lemma}{theorem}
\newtheorem{lemma}[lemma]{Lemma}
\aliascntresetthe{lemma}

\newaliascnt{proposition}{theorem}
\newtheorem{proposition}[proposition]{Proposition}
\aliascntresetthe{proposition}

\newaliascnt{corollary}{theorem}

\aliascntresetthe{corollary}

\newaliascnt{conjecture}{theorem}

\aliascntresetthe{conjecture}

\newaliascnt{openQ}{theorem}

\aliascntresetthe{openQ}

\newaliascnt{quest}{theorem}

\aliascntresetthe{quest}

\newaliascnt{questx}{conjx}

\aliascntresetthe{questx}

\theoremstyle{definition}

\newaliascnt{defn}{theorem}

\aliascntresetthe{defn}

\newaliascnt{example}{theorem}

\aliascntresetthe{example}

\newaliascnt{rem}{theorem}

\aliascntresetthe{rem}

\makeatletter
\def\tagform@#1{\maketag@@@{\ignorespaces#1\unskip\@@italiccorr}}
\let\orgtheequation\theequation
\def\theequation{(\orgtheequation)}
\makeatother
\def\equationautorefname~{}

\DeclareMathAlphabet{\mathpzc}{OT1}{pzc}{m}{it}
\newcommand{\e}{\varepsilon}
\newcommand{\B}{{\mathbb B}}
\newcommand{\Bn}{{{\mathbb B}^n}}
\newcommand{\D}{\mathbb{D}}
\newcommand{\Jac}{\operatorname{J}}
\newcommand{\R}{{\mathbb R}}
\newcommand{\Rn}{{{\mathbb R}^n}}
\newcommand{\vm}[1]{\left|#1\right|}
\newcommand{\mt}{\mathfrak{g}}
\newcommand{\dlth}{\Delta_{hyp}}
\newcommand{\dlts}{\Delta_{\theta}}
\newcommand{\tsp}{\Theta}
\newcommand{\arctanh}{\operatorname{arctanh}}

\begin{document}
\title[Third Neumann eigenvalue]{Two balls maximize the third Neumann eigenvalue in hyperbolic space}


	\begin{abstract}
We show that the third eigenvalue of the Neumann Laplacian in hyperbolic space is maximal for the disjoint union of two geodesic balls, among domains of given volume.
This extends a recent result by Bucur and Henrot in Euclidean space, while providing a new proof of a key step in their argument
	\end{abstract}
	
\author[]{P. Freitas and R. S. Laugesen}
\address{Departamento de Matem\'atica, Instituto Superior T\'ecnico, Universidade de Lisboa, Av. Rovisco Pais 1,
P-1049-001 Lisboa, Portugal {\rm and}
Grupo de F\'isica Matem\'atica, Faculdade de Ci\^encias, Universidade de Lisboa,
Campo Grande, Edif\'icio C6, P-1749-016 Lisboa, Portugal}
\email{psfreitas@fc.ul.pt}
\address{Department of Mathematics, University of Illinois, Urbana,
	IL 61801, U.S.A.}
\email{Laugesen@illinois.edu}

	\maketitle	

\section{\bf Introduction and results} \label{sec:intro}

Which shape maximizes the second eigenvalue of the Neumann Laplacian among Euclidean regions of given volume? This question, which arises naturally in the wake of the Rayleigh--Faber--Krahn
result for the Dirichlet Laplacian, was answered in the 1950s by Szeg\H{o}~\cite{S54} for two-dimensional simply--connected domains and by Weinberger~\cite{W56} for the general case in all
dimensions, the answer being the ball. 

Which shape maximizes the third eigenvalue of the Neumann Laplacian? This time the answer is a disjoint union of two equal
balls. The proof of this result had to wait for more than 50 years, and it was
again proved first for simply--connected domains in the plane, by Girouard, Nadirashvili and Polterovich in 2009~\cite{GNP09}, and in full generality by Bucur and Henrot in 2019~\cite{BH19}. 

\subsection{Hyperbolic results} In this paper we extend the latter result to hyperbolic space, showing that two disjoint geodesic balls of the same volume maximize the third Neumann
eigenvalue of the Laplace--Beltrami operator among regions of given volume. More precisely, our main result is the following, in hyperbolic spaces of dimension $n$ greater than or equal to $2$.
\begin{thmx}[Third hyperbolic Neumann eigenvalue is maximal for two balls]~\label{th:hyperbolic2ball}\ 
 Let $\Omega$ be a bounded open set with Lipschitz boundary in the hyperbolic space $\mathbb{H}^{n}$ of constant negative curvature $\kappa$, and
 denote the eigenvalues of the Neumann Laplacian in $\Omega$ by $0=\eta_{1}\leq \eta_{2} \leq \eta_{3} \leq \cdots.$ Then
 \[
   \eta_3(\Omega) \leq \eta_3(B \sqcup B),
 \]
 where $B$ is a ball having half the hyperbolic volume of $\Omega$. Equality holds if and only if $\Omega$ is a disjoint union of two balls 
 in $\mathbb{H}^{n}$, of equal hyperbolic volume. 
\end{thmx}

As in Bucur and Henrot's proof in the Euclidean setting, our approach builds on Weinberger's method for the second eigenvalue. In his proof,
the required orthogonality is to the first (constant) eigenfunction, and thus amounts to obtaining trial functions with zero average.
For the third eigenvalue, the trial functions must have zero average and also be orthogonal to the first excited state. This construction is made possible, at its heart, by the fact that the eigenvalue $\eta_2(B)$ of the ball has multiplicity $n$. More precisely, $\eta_3(B \sqcup B)$ has multiplicity $2n$ because the spectrum of a disjoint union is simply the union of the spectra and so the first three eigenvalues of the disjoint union of two balls satisfy
\[
0 = \eta_1(B \sqcup B) = \eta_2(B \sqcup B) < \eta_3(B \sqcup B) = \eta_2(B) .
\]
This multiplicity will be instrumental in constructing trial functions that satisfy the required orthogonality conditions. Further, we provide a new way of identifying such trial functions. This method also works in Euclidean space, yielding a new proof of the central step in Bucur and Henrot's result.

Another key ingredient in our proof is a degree theory computation of Petrides~\cite{P14} for maps between spheres, and here too we present
what we believe to be a simpler, global proof. The proof of~\autoref{th:hyperbolic2ball} is then covered in the second part of the paper, beginning
in~\autoref{sec:hyperbolic2ballproof}. 

We recall that the hyperbolic analogue of the Szeg\H{o}--Weinberger result for the second eigenvalue, asserting maximality of the ball under a hyperbolic volume constraint, was proved by Bandle \cite{B72,B80} for $2$-dimensional simply connected surfaces, with the general higher dimensional case later mentioned by Chavel~\cite[p.\ 80] {C80}, \cite[pp.\,43--44]{C84}. Detailed expositions were provided by Ashbaugh
and Benguria \cite[{\S}6]{AB95} and Xu \cite[Theorem 1]{X95}, and the requisite center of mass lemmas can be found in Benguria and Linde \cite[Theorem 6.1]{BL07} and Laugesen~\cite[Corollary 2]{L20b}.

\subsection{Euclidean results}  Since our methods apply both to Euclidean and hyperbolic spaces, we first carry out the proof in the former case, as this corresponds to a simpler set-up and, we believe, makes the general approach clearer.

Write $\B=\Bn$ for the unit ball in $\Rn$, where $n \geq 2$. Let $\Omega \subset \Rn$ be a bounded open set with Euclidean volume $\vm{\Omega}$ and Lipschitz boundary. The Lipschitz requirement restricts $\Omega$ to having finitely many components. The Neumann
eigenvalue problem for the Laplacian is
\begin{equation*}\label{neumannproblem}
\begin{split}
- \Delta u & = \mu u \ \quad \text{in $\Omega$} \\
\frac{\partial u}{\partial\nu} & = 0 \qquad \text{on $\partial \Omega$} 
\end{split}
\end{equation*}
and, under the above conditions, its spectrum is discrete with eigenvalues satisfying 
\[
0 = \mu_1 \leq \mu_2 \leq \mu_3 \leq \cdots \to \infty .
\] 
The maximization result for $\mu_{3}$, obtained by Bucur and Henrot in~\cite{BH19}, is the following.
\begin{thmx}[Bucur and Henrot: third Neumann eigenvalue is maximal for two balls] \label{th:Euclidean2ball}
If $\Omega \subset \Rn$ is a bounded open set with Lipschitz boundary then
\begin{equation*}\label{ineq:main}
  \mu_3(\Omega) \vm{\Omega}^{2/n} \leq \mu_3(\B\sqcup\B) (2\vm{\B})^{2/n}.
\end{equation*}
Equality holds if and only if $\Omega$ is a disjoint union of two balls of equal volume. 
\end{thmx}

One of the goals of this paper is to present a proof of \autoref{th:Euclidean2ball} that is different at its homotopic core from
Bucur and Henrot's proof. That argument begins in \autoref{sec:Euclidean2ballproof}, and a discussion of the similarities and
differences between the two proofs is provided in~\autoref{sec:simdiff}.

\subsubsection*{Remarks on the literature.}
Girouard, Nadirashvili and Polterovich~\cite[Theorem 1.1.3]{GNP09} first proved the result on the third eigenvalue for simply connected domains in the plane, and Girouard and Polterovich \cite[Theorem 1.7]{GP10} extended the argument to surfaces with variable nonpositive curvature.
Their Neumann trial functions in the unit disk were modified to Euclidean space by Bucur and Henrot \cite{BH19}, and adapted to Robin eigenvalues in the planar
case by Girouard and Laugesen \cite{GL19}. 

Girouard, Nadirashvili and Polterovich worked in terms of  measures folded across hyperplanes (which in their context meant hyperbolic geodesics in the disk), and 
they maximized over $2$-dimensional spaces of trial functions. Girouard and Laugesen clarified the construction by composing trial functions with a fold map in order 
to obtain trial functions that are even across the hyperplane, and they used uniqueness of the center of mass point to reduce from a $2$-dimensional trial space to 
a single trial function. This function depends on a parameter lying on the circle, and so the degree theory required to finish the proof is a straightforward winding 
number argument. 

Bucur and Henrot worked in the Euclidean context in all dimensions, and employed a different parameterization for what is essentially the same family of trial functions. 
The degree theory required to finish their proof in higher dimensions is more difficult: they finished with an ingenious two-step homotopy argument for their vector 
field in $\R^{2n}$. 

One contribution of the current paper is what we believe to be a conceptually simpler proof of Bucur and Henrot's ``two ball''  \autoref{th:Euclidean2ball}. We adapt the 
parameterization of Girouard, Nadirashvili and Polterovich to Euclidean space in all dimensions, and rely on uniqueness of
the center of mass point together with a beautiful degree theory result of Petrides \cite{P14} for maps between spheres (\autoref{th:degree} below). 

A ``relaxed'' version of the eigenvalue maximization problem, in which the indicator function of the region $\Omega$ is replaced by a weight function, was treated by Bucur 
and Henrot \cite[Theorem 3]{BH19}. Their weight could have unbounded support. We shall not pursue such generalizations here.

\subsection{Future directions} Now that the third Neumann eigenvalue is known to be maximal for two disjoint balls in spaces of
non-positive constant curvature (Euclidean and hyperbolic), it is natural to ask the same question in the positive curvature case,
that is, for domains in the standard round sphere. The analogue of Weinberger's result for the second eigenvalue is known
to hold for domains contained in a hemisphere and certain other domains~\cite{AB95}, but is not known to hold in general. Thus some restriction on the domain might be needed to get a result on the third eigenvalue. 

Another extension would be to higher eigenvalues: one might ask whether the fourth eigenvalue is maximal for three disjoint balls.
This turns out not to be the case. Numerical work by Antunes and Freitas~\cite[Section 4]{AF12} (see also the chapter by Antunes and Oudet in~\cite{H17}) suggests the maximizer for $\mu_4$ among Euclidean domains of given area is not a union of
disks, but rather a domain that looks somewhat like three touching disks with the joining regions smoothed out. 

Averaging the eigenvalues can improve their behavior and lead to a positive result. Notably, the harmonic mean of $\mu_2,\dots,\mu_n$ is maximal for the ball, among domains in $\Rn$ of given
volume, by a recent result of Wang and Xia \cite[Theorem 1.1]{WX20} that directly strengthens Weinberger's theorem for $\mu_2$. They also prove the analogous result for domains in hyperbolic space \cite[Theorem 1.2]{WX20}. Thus, for example, Wang and Xia's result implies the harmonic mean of $\mu_2$ and $\mu_3$ is maximal for the ball, among domains of given volume in $n$ dimensions ($n \geq 3$), while Bucur and Henrot's theorem in euclidean space and ours in hyperbolic space say that $\mu_3$ by itself is maximal for the disjoint union of two balls. Looking now to the future, it is an open problem whether the spectral zeta function $\sum_{j=2}^\infty \mu_j^{-s}$ is minimal for the ball, when $s>n/2$. If true, it would in a sense extend Wang and Xia's result to the full spectrum. 

Finally, a further avenue of investigation would be an extension to the Robin problem, in line with our results in~\cite{FL20,FL18a}. In these papers it
was shown that it is possible to connect the Neumann and Steklov isoperimetric inequalities for the second eigenvalue, via the Robin parameter. However, numerical work for the third eigenvalue 
points in the direction that the extremal sets change as soon as the Robin parameter becomes negative, with, in this case, the extremal domain becoming connected~\cite{AFK17}.

For more open problems in spectral shape optimization, we warmly recommend a book edited by Henrot \cite{H17}. For developments on
the related problem of maximizing eigenvalues over conformal classes on surfaces, see work of Karpukhin, Nadirashvili, Penskoi and
Polterovich~\cite{KNPP18,KNPP20}, Karpukhin and Stern~\cite{KS20}, and Petrides \cite{P18}.

\section{\bf The Petrides theorem on the degree of a map with reflection symmetry}

Our trial function construction for the third Neumann eigenvalue will depend on a topological theorem due to Petrides. His result, which is similar to the Borsuk--Ulam
theorem for odd mappings, says that if a map from the sphere to itself has a certain reflection symmetry property then its degree must be nonzero. 

This fact will play a role in the paper analogous to that of the Brouwer fixed point theorem (or no-retraction theorem) in Szeg\H{o} and Weinberger's papers for the
second Neumann eigenvalue. Namely, the topological result will be used in \autoref{pr:vanishing} to show existence of a trial function that is orthogonal to the first
two Neumann eigenfunctions. 

We will give a new proof for Petrides's theorem. His proof was local in nature, whereas the one below is global. 

Given a unit vector $p$ in Euclidean space, write
$R_p$ for the reflection with respect to the hyperplane through the origin that is perpendicular to $p$:
\[
R_p(y) = y - 2(y \cdot p) p .
\]
\begin{theorem}[Reflection symmetry implies nonzero degree; Petrides \protect{\cite[Claim 3]{P14}}] \label{th:degree}
Assume $\phi : S^m \to S^m$ is continuous, for some $m \geq 1$. If $\phi$ satisfies the reflection symmetry property 
\begin{equation} \label{eq:reflsymm}
\phi(-p) = R_p \big( \phi(p) \big) , \qquad p \in S^m ,
\end{equation}
then $\phi$ has nonzero degree, meaning $\phi$ is not homotopic to a constant map. More precisely, if $m$ is odd then $\deg(\phi)=1$, and if $m$ is even then $\deg(\phi)$ is odd.
\end{theorem}
First we need an elementary result about maps that are homotopic to the identity.
\begin{lemma} \label{le:homotopy}
Suppose $\psi : S^m \to S^m$ is continuous, where $m \geq 1$. If $\psi(p) \cdot p \geq 0$ for all $p \in S^m$ then $\deg(\psi)=1$. If $\psi(p) \cdot p \leq 0$ for all $p \in S^m$ then $\deg(\psi)=(-1)^{m+1}$. 
\end{lemma}
The intersection of the two cases, where $\psi(p) \cdot p = 0$ for all $p$ (``a hairy ball''), can obviously occur only for odd $m$, in which case the degree is $1$.
\begin{proof}
Suppose $\psi(p) \cdot p \geq 0$ for all $p$. Then $\psi$ is homotopic to the identity via
\[
\Psi(p,t) = \frac{(1-t)\psi(p)+tp}{|(1-t)\psi(p)+tp|} , \qquad p \in S^m , \quad t \in [0,1] ,
\]
where $\Psi(p,0)=\psi(p), \Psi(p,1)=p$, and the numerator is nonzero because its dot product with $p$ is $(1-t) \psi(p) \cdot p + t$, which is positive when $t \in (0,1]$. The identity map has degree $1$, and hence so does $\psi$.

If $\psi(p) \cdot p \leq 0$ for all $p$, then $\deg(-\psi)=1$ by the case just proved, and so $\deg(\psi)=(-1)^{m+1}$.
\end{proof}
\begin{proof}[Proof of \autoref{th:degree}]
The case of the circle ($m=1$) admits a simple winding number proof, and so we give that argument first. Regarding $p \in S^1$ as a complex number, the reflection formula becomes $R_p(z)= -p^2\overline{z}$ for complex numbers $z$. Thus the reflection symmetry hypothesis \eqref{eq:reflsymm} implies that  
\[
\phi(-p)\phi(p)=-p^2 \overline{\phi(p)} \phi(p)=-p^2 . 
\]
The argument of the right side increases by $4\pi$ as $p$ goes once around the circle. On the left side, the arguments of $\phi(p)$ and $\phi(-p)$ increase by the same amount as each other,
and hence must increase by $2\pi$, implying that $\phi$ has degree $1$. 

\smallskip
Now we prove the theorem for all $m \geq 1$. 

\smallskip \emph{Step 1 --- Reducing to a smooth $\phi$ whose normal component changes sign.} 
By an argument of Petrides \cite[p.\,2391]{P14} it is possible to reduce the theorem to the case of $\phi$ smooth. This involves using successive smoothing on $m+1$ almost-hemispheres centered on the coordinate axes, with the approximations extended to complementary hemispheres via the reflection
symmetry formula. So, from now on, $\phi$ is assumed to be smooth.

For $s\in[0,1)$ define the level sets 
\[
A(s) = \{ p \in S^m : \phi(p) \cdot p > s \} , \qquad B(s) = \{ p \in S^m : \phi(p) \cdot p < s \}
\]
and  
\[
Z(s) = \{ p \in S^m : \phi(p) \cdot p = s \} .
\]
These sets are invariant under the antipodal map, meaning $A(s) = -A(s)$, $B(s) = -B(s)$ and $Z(s) = -Z(s)$, because the normal component is even:
\[
\phi(p) \cdot p = R_p \left(\phi(p)\right) \cdot R_p\left( p\right) = \phi(-p) \cdot (-p) 
\]
by the reflection symmetry hypothesis. 
We may further suppose $\phi(p) \cdot p$ is nonconstant, for otherwise \autoref{th:degree} follows immediately from \autoref{le:homotopy}.

\smallskip \emph{Step 2 --- Reducing $Z$ to a smooth submanifold.} Choose a regular value $s \in [0,1)$ for the function $\phi(p) \cdot p$,
so that the level set $Z(s)$ is an imbedded submanifold in the sphere that forms the boundary of both the
superlevel set $A(s)$ and the sublevel set $B(s)$. Define a homotopy 
\[
\Phi(p,t) = \frac{\phi(p)-tp}{|\phi(p)-tp|} , \qquad p \in S^m , \quad t \in [0,s] ,
\]
and note that $\Phi(p,0)=\phi(p)$ and the numerator is nonzero for each $t$ because $\phi(p)$ is a unit vector while $|tp|=t \leq s < 1$. Denote the right endpoint map
of the homotopy by $\psi(p)=\Phi(p,s)$, so that $\psi$ is a smooth map of the sphere to itself with $\deg(\psi)=\deg(\phi)$. Observe also that $\psi$ satisfies the
reflection symmetry condition \eqref{eq:reflsymm}, because both $\phi(p)$ and $p$ satisfy it. Thus it suffices to prove the theorem for $\psi$.

The zero superlevel set for $\psi$ is 
\[
A_\psi(0) = \{ p \in S^m : \psi(p) \cdot p > 0 \} = \{ p \in S^m : \phi(p) \cdot p > s \} = A(s) ,
\]
and similarly for its sublevel and zero sets $B_\psi(0) = B(s)$ and $Z_\psi(0)=Z(s)$, respectively. Thus, by changing the name of $\psi$ to $\phi$ and dropping the $\psi$ index, we may suppose from now on that $\phi$ is smooth and $Z=Z(0)$ is an imbedded submanifold in the sphere that forms the boundary of both $A=A(0)$ and $B=B(0)$. That is, $Z$ is locally the graph of a smooth function with $A$ on one side of the graph and $B$ on the other side. 

\smallskip \emph{Step 3 --- Comparison map.} The degree of $\phi$ will be compared with the degree of the map $\widetilde{\phi} : S^m \to S^m$ defined by 
\[
\widetilde{\phi}(p) 
= 
\begin{cases}
\phi(p) \\
\phi(-p) 
\end{cases}
\hspace{-12pt} = 
\begin{cases}
\phi(p) & \text{when $p \in A \cup Z$,} \\
R_p \, \phi(p) & \text{when $p \in B \cup Z$,}
\end{cases}
\]
where the reflection symmetry assumption \eqref{eq:reflsymm} has been used on $B \cup Z$; notice the definition is consistent on $Z$ since $\phi(p) \cdot p = 0$ there and so $\phi(p)=R_p\, \phi(p)$ on $Z$. 

This map $\widetilde{\phi}$ is piecewise smooth, and $\deg(\widetilde{\phi})=1$ by \autoref{le:homotopy}, because $\widetilde{\phi}(p) \cdot p \geq 0$ for all $p \in S^m$. 

\smallskip \emph{Step 4 --- Odd dimensions.} Suppose $m$ is odd. Write $\Jac[\phi]$ for the determinant of the derivative map of $\phi$, so that  
\begin{align}
\deg(\phi ) 
& = \fint_{S^m} \Jac[\phi](p) \, dS(p) \notag \\
& = \frac{1}{|S^m|} \int_A \Jac[\phi](p) \, dS(p) + \frac{1}{|S^m|} \int_B \Jac[\phi](p) \, dS(p) , \label{eq:magic}
\end{align}
where we need not integrate over $Z$ because that set has measure zero. In the integral over $A$, we may replace $\phi$ with $\widetilde{\phi}$. In the integral over $B$, write $N(p)=-p$ for the antipodal map, and change variable to obtain that  
\begin{align*}
\int_B \Jac[\phi](p) \, dS(p)
& = \int_{N^{-1}(B)} (\Jac[\phi] \circ N)(p) \, dS(p)  \\
& = \int_{N^{-1}(B)} \Jac[\phi \circ N](p) \, dS(p) 
\end{align*}
by the chain rule, because $\Jac[N]=(-1)^{m+1}=1$ when $m$ is odd. Next, since $N^{-1}(B)=-B=B$, and $\phi \circ N = \widetilde{\phi}$ on $B$, we conclude
\[
\int_B \Jac[\phi](p) \, dS(p) = \int_B \Jac[\widetilde{\phi}](p) \, dS(p) .
\]
Thus we may replace $\phi$ with $\widetilde{\phi}$ in the second integral of \eqref{eq:magic} also, and so $\deg(\phi)=\deg(\widetilde{\phi})=1$.

\smallskip \emph{Step 5 --- Even dimensions.} Suppose $m$ is even. By arguing as in the previous step, except this time using that $\Jac[N]=(-1)^{m+1}=-1$ because $m$ is even, we find
\begin{equation} \label{eq:degdecomp}
\deg(\phi) 
= \frac{1}{|S^m|} \int_A \Jac[\phi](p) \, dS(p) - \frac{1}{|S^m|} \int_B \Jac[\widetilde{\phi}](p) \, dS(p) .
\end{equation}
Obviously also 
\[
\deg(\widetilde{\phi}) 
= \frac{1}{|S^m|} \int_A \Jac[\phi](p) \, dS(p) + \frac{1}{|S^m|} \int_B \Jac[\widetilde{\phi}](p) \, dS(p) .
\]
Since $\deg(\widetilde{\phi}) = 1$, adding the last two equations shows that 
\[
\deg(\phi) + 1 = 2 \, \frac{1}{|S^m|} \int_A \Jac[\phi](p) \, dS(p) . 
\]
Thus to show $\phi$ has odd degree, we want to show the expression 
\begin{equation} \label{eq:alphadegree}
\frac{1}{|S^m|} \int_A \Jac[\phi](p) \, dS(p)
\end{equation}
is an integer. We accomplish this by showing it equals the degree of a mapping. 

We regard the closure $\overline{A} = A \cup Z$ as a compact manifold without boundary, by identifying antipodal points in $Z$. In more detail, given a point $p \in Z$ and any small $\e>0$, one can form a neighborhood of $p$ in $\overline{A}$ by gluing together the sets 
\[
\{ q \in A : |q-p| < \e \} \quad \text{and} \quad \{ q \in A : |q-(-p)| < \e \} 
\]
along their common interface, which is the smooth submanifold 
\[
\{ q \in Z : |q-p| < \e \} \stackrel{N}{\sim} \{ q \in Z : |q-(-p)| < \e \} ,
\]
where $\stackrel{N}{\sim}$ denotes the antipodal equivalence relation on $Z$. A little thought shows that the resulting manifold $\overline{A}$ is orientable, because $m$ is even. This manifold might not be connected, but it can have only finitely many components, since $Z$ is a smooth submanifold of the sphere. 

Define a piecewise smooth map $\alpha : \overline{A} \to S^m$ by restriction: let $\alpha = \phi \,|_{\, \overline{A}}$, noting the restriction is consistent under the antipodal identification on $Z$ because 
\[
p \in Z \quad \Longrightarrow \quad \phi(p) \cdot p = 0 \quad \Longrightarrow \quad \phi(p) = R_p \, \phi(p) \quad \Longrightarrow \quad \phi(p) = \phi(-p) ,
\]
by the reflection symmetry hypothesis. The degree of the map $\alpha$ is exactly the expression in \eqref{eq:alphadegree} (by the de Rham approach to degree theory \cite[Corollary III.2.4]{OR09}), which means expression \eqref{eq:alphadegree} must be an integer, completing the proof. 

\smallskip
\noindent \emph{Remark.} A more ``symmetrical'' formulation of Step 5, when $m$ is even, arises from regarding both $\overline{A}$ and $\overline{B}$ as compact orientable manifolds without boundary, via the antipodal identification on $Z$. Define $\beta : \overline{B} \to S^m$ by restriction as $\beta = \phi \,|_{\, \overline{B}}$, so that $\deg(\phi) = \deg(\alpha) - \deg(\beta \circ N)$ by formula \eqref{eq:degdecomp}, where $\alpha : \overline{A} \to S^m$ and $\beta \circ N : \overline{B} \to S^m$. Since also $1=\deg(\widetilde{\phi}) = \deg(\alpha)+\deg(\beta \circ N)$ by definition of $\widetilde{\phi}$, adding the two equations shows that $\deg(\phi)+1=2\deg(\alpha)$, and hence $\deg(\phi)$ is odd.

\smallskip
\noindent \emph{Remark.} 
The nonvanishing of $\deg(\phi)$ can be proved in a different way when $m$ is odd. For suppose to the contrary that $\phi$ is homotopic to a constant map, say to $p \mapsto e_1$ where $e_1$ is the unit vector in the first coordinate direction. Then $\phi(-p)$ is also homotopic to the constant map, while $R_p \big( \phi(p) \big)$ is homotopic to the map $p \mapsto R_p(e_1)$. After joining these homotopies by the reflection symmetry hypothesis \eqref{eq:reflsymm}, we conclude that $p \mapsto R_p(e_1)$ is homotopic to a constant. But $R_p(e_1)$ has nonzero degree when $m$ is odd, by an observation of Girouard, Nadirashvili and Polterovich \cite[p.\,656]{GNP09}. This contradiction shows that $\phi$ has nonzero degree. 
\end{proof}

\section{\bf An elementary lemma related to trial functions}

The following lemma is needed in both the Euclidean and hyperbolic parts of the paper, when trial functions of the form $g(r)x_j/r$ are employed. The idea is due to Weinberger \cite{W56}. Write $|x|=r$.
\begin{lemma}\label{le:trialfn}
Suppose $w^*$ and $w_*$ are Borel measures on $\Rn, n \geq 2$, and $g(r)$ is $C^1$-smooth for $r \geq 0$, with $g(0)=0$. If $\mu \in \R$ satisfies 
\[
\mu \leq \frac{\int_\Rn |\nabla \big( g(r)x_j/r \big)|^2 \, dw^*(x)}{\int_\Rn \big( g(r)x_j/r \big)^2 \, dw_*(x)}
\]
for $j=1,\dots,n$, and each integral is positive and finite, then 
\[
\mu \leq \frac{\int_\Rn \left( g^\prime(r)^2 + \frac{n-1}{r^2} g(r)^2 \right) dw^*(x)}{\int_\Rn g(r)^2 \, dw_*(x)} .
\]
\end{lemma}
\begin{proof}
Clearing the denominator and summing over $j$ gives that 
\[
\mu \int_\Rn g(r)^2 \sum_{j=1}^n \big( x_j/r \big)^2 \, dw_*(x) \leq \int_\Rn \sum_{j=1}^n |\nabla \big( g(r)x_j/r \big)|^2 \, dw^*(x) .
\]
The sum on the left simplifies to $1$, giving the desired denominator. The gradient on the right can be computed straightforwardly, obtaining that
\begin{equation} \label{eq:threestar}
\sum_{j=1}^n |\nabla \big( g(r)x_j/r \big)|^2 = g^\prime(r)^2 + \frac{n-1}{r^2} g(r)^2 ,
\end{equation}
which is what we want for the numerator. 
\end{proof}

\section{\bf Weighted mass transplantation}  \label{sec:transplantation}	

A minor adaptation of Weinberger's mass transplantation method will be needed later. Write 
\[
V_\omega(\Omega) = \int_\Omega \omega(r) \, dx
\]
for the volume of $\Omega$ with respect to a radial weight $\omega(r)$.
\begin{lemma}[Mass transplantation]\label{transplantation}
Suppose $\omega(r)$ is a positive, Lebesgue measurable, radial weight function on $\Rn$, and $\Omega_L$ and $\Omega_U$ are measurable sets in $\Rn$ with finite weighted volume adding up to twice the weighted volume of a ball $B$ centered at the origin: 
\begin{equation} \label{eq:volumesum}
V_\omega(\Omega_L)+V_\omega(\Omega_U)=2V_\omega(B) < \infty .
\end{equation}
If $h : [0,\infty) \to \R$ is decreasing then 
\begin{equation} 
\label{eq:monot}
\big( \int_{\Omega_L} + \int_{\Omega_U} \big) h(r) \omega(r) \, dx \leq 2 \int_B h(r) \omega(r) \, dx .
\end{equation}
If $h(r)$ is strictly decreasing then equality holds if and only if $\Omega_L = \Omega_U = B$ up to sets of measure zero. 
\end{lemma}
\begin{proof}
The volume hypothesis \eqref{eq:volumesum} implies that 
\begin{equation} \label{eq:totalvolume}
V_\omega(\Omega_L \setminus B) + V_\omega(\Omega_U \setminus B) = V_\omega(B \setminus \Omega_L) + V_\omega(B \setminus \Omega_U) .
\end{equation}
Since $h$ is radially decreasing, we find
\begin{align}
& \big( \int_{\Omega_L} + \int_{\Omega_U} \big)  h(r) \omega(r) \, dx \notag \\
& = \big( \int_{\Omega_L \cap B} + \int_{\Omega_U \cap B} \big) h(r) \omega(r) \, dx + \big( \int_{\Omega_L \setminus B} + \int_{\Omega_U \setminus B} \big)  h(r) \omega(r) \, dx \notag \\
& \leq \big( \int_{\Omega_L \cap B} + \int_{\Omega_U \cap B} \big) h(r) \omega(r) \, dx + \big( V_\omega(\Omega_L \setminus B) + V_\omega(\Omega_U \setminus B) \big) h(1) \label{massproof1} \\
& = \big( \int_{B \cap \Omega_L} + \int_{B \cap \Omega_U} \big) h(r) \omega(r) \, dx + \big( V_\omega(B \setminus \Omega_L) + V_\omega(B \setminus \Omega_U) \big) h(1) \qquad \text{by \eqref{eq:totalvolume}} \notag \\
& \leq \big( \int_{B \cap \Omega_L} + \int_{B \cap \Omega_U} \big) h(r) \omega(r) \, dx + \big( \int_{B \setminus \Omega_L} + \int_{B \setminus \Omega_U} \big) h(r) \omega(r) \, dx \label{massproof2} \\
& = 2 \int_B h(r) \omega(r) \, dx < \infty , \notag
\end{align}
which proves the inequality \eqref{eq:monot} in the proposition. 

Suppose $h$ is strictly decreasing and equality holds in \eqref{eq:monot}. Then equality must hold in \eqref{massproof1}, which forces the sets $\Omega_L \setminus B$ and $\Omega_U \setminus B$ to have weighted volume zero (using here that $h$ is strictly decreasing), and hence to have measure zero, since $\omega>0$. Similarly, equality must hold in \eqref{massproof2}, and so $B \setminus \Omega_L$ and $B \setminus \Omega_U$ have measure zero. 
\end{proof}
%
%
%

\section{\bf Proof of \autoref{th:Euclidean2ball} (Euclidean) --- constructing the trial functions}  \label{sec:Euclidean2ballproof}	

The idea of constructing trial functions from the eigenfunctions of the ball goes back to Szeg\H{o} \cite{S54} and Weinberger \cite{W56}, in their work on the second eigenvalue. Girouard,
Nadirashvili and Polterovich \cite{GNP09} folded these trial functions across a hyperplane in order to obtain trial functions for the third eigenvalue. (Technically, they folded the domain
rather than the trial functions, following an idea of Nadirashvili \cite{N02} on the sphere.) They worked in the disk in $2$ dimensions with hyperbolic reflections, but the construction adapts
to Euclidean space too, as developed below. Bucur and Henrot \cite{BH19} constructed the same family of trial functions by gluing rather than folding, and they parameterized the trial functions
somewhat differently. These matters are discussed in more detail at the end of this section.

\subsection*{Construction of trial functions} 
By scale invariance we may assume
\[
\vm{\Omega} = 2\vm{\B} ,
\]
so that $\Omega$ has twice the volume of the unit ball. 

Let $g(r)$ be the radial part of the second Neumann eigenfunction for the unit ball, extended to be constant for $r \geq 1$. That is, $g(0)=0$ and 
\[
g(r) = 
\begin{cases}
r^{1-n/2}J_{n/2}(k_{n/2}^\prime r) , & 0 < r \leq 1 , \\
g(1) , & r > 1 ,
\end{cases}
\]
where $k_{n/2}^\prime$ is the first positive zero of $(r^{1-n/2}J_{n/2}(r))^\prime$ and $J_{n/2}$ is the Bessel function of order $n/2$. The eigenvalue of the ball is $\mu_2(\B) = (k_{n/2}^\prime)^2$. For example, in dimension $n=2$ one gets $k_1^\prime \simeq 1.84$. Notice $g^\prime=0$ at $r=1$, from both the left and right. 

What we need to know about the ball eigenfunction $g(r)y_j/r$ (where $r=|y|$ and $j=1,\ldots,n$) is that $g$ is continuous and increasing for $r \geq 0$, and positive for $r > 0$, and $g^\prime$ is continuous with $g^\prime(1)=0$. \autoref{basic2Euclidean} at the end of the paper provides a precise statement and references for these properties. 

Define $v : \Rn \to \Rn$ to be the vector field
\begin{equation} \label{eq:vdef}
v(y) = g(|y|) \frac{y}{|y|} , \qquad y \in \Rn \setminus \{ 0 \} ,
\end{equation}
with $v(0)=0$. Note $v$ is continuous everywhere, including at the origin since $g(0)=0$. 

Each component $v_j(y)=g(r)y_j/r$ of the vector field is a Neumann eigenfunction on the unit ball. The eigenfunction equation $-\Delta v_j = \mu_2(\B)v_j$ implies that $g$ satisfies the Bessel-type equation
\begin{equation}\label{eq:besseltype}
- g''(r) = \frac{n-1}{r}g'(r) + \left( \mu_2(\B) - \frac{n-1}{r^2} \right) g(r) .
\end{equation}

A useful normalization is that by translating $\Omega$ we may suppose the ``Weinberger point'' (or $g$-center of mass) of $\Omega$ lies at the origin:
\begin{equation} \label{eq:centroid}
\int_\Omega v(y) \, dy = 0 .
\end{equation}
The existence of such a translation was proved by Weinberger \cite{W56}, using Brouwer's fixed point theorem, although note the set-up is slightly different here because our $\Omega$ has twice the volume of the ball to which $g$ is adapted. Incidentally, an existence proof by energy minimization was explored recently by
Laugesen~\cite[Corollary 2]{L20a} (take $f \equiv 1$ there), where references to earlier work can be found.

The trial functions will be obtained by folding or reflecting $v$ across hyperplanes, thus constructing functions that are even with respect to the hyperplane. Let 
\[
H = H_{p,t} = \{ y \in \Rn :  y \cdot p < t \} , \qquad p \in S^{n-1}, \quad t \geq 0 ,
\]
to be the open halfspace with normal vector $p$ and ``height'' $t$, and let
\[
R_H(y) \equiv R_{p,t}(y) = y + 2(t - y \cdot p) p
\]
be reflection in the hyperplane $\partial H$. Define the ``fold map'' onto the closed halfspace $\overline{H}$ by 
\[
F_H(y) \equiv F_{p,t}(y) = 
\begin{cases}
y & \text{if\ } y \in H, \\
R_H(y) & \text{if\ } y \in \Rn \setminus H .
\end{cases}
\]
We will use the notation $R_H$ and $F_{H}$ whenever this is unambiguous, and $R_{p,t}$ and $F_{p,t}$ whenever it is necessary to refer to any of the arguments explicitly.

Our trial functions on $\Omega$ will be the $n$ components of the vector field
\[
y \mapsto v(F_H(y)-c) ,
\]
where $c \in \Rn$. To visualize these trial functions, imagine centering the vector field at $c$ and then replace its values in the complement of $H$ with the 
even extension of the vector field across $\partial H$. The resulting vector field belongs to the Sobolev space $H^1(\Omega;\Rn)$, since it is continuously 
differentiable on each side of the hyperplane $\partial H$ and is continuous across the hyperplane. 

\smallskip \noindent \emph{Remark.}
The number of parameters matches the number of conditions to be satisfied, because the parameters $(p,t,c)$ lie in the $2n$-dimensional space
$S^{n-1} \times \R \times \Rn$, and we aim to satisfy $2n$ orthogonality conditions, namely we want each of the $n$ trial functions to be orthogonal to the
first and second eigenfunctions on $\Omega$. Thus the set-up considered is dimensionally compatible.

The different possibilities for the trial functions are illustrated in Figure~\ref{fig:trialfns}, for the special case in $2$ dimensions where the hyperplane is the vertical axis and $H$ is the left halfspace. 

\begin{figure}
\begin{center}
\includegraphics[scale=0.55]{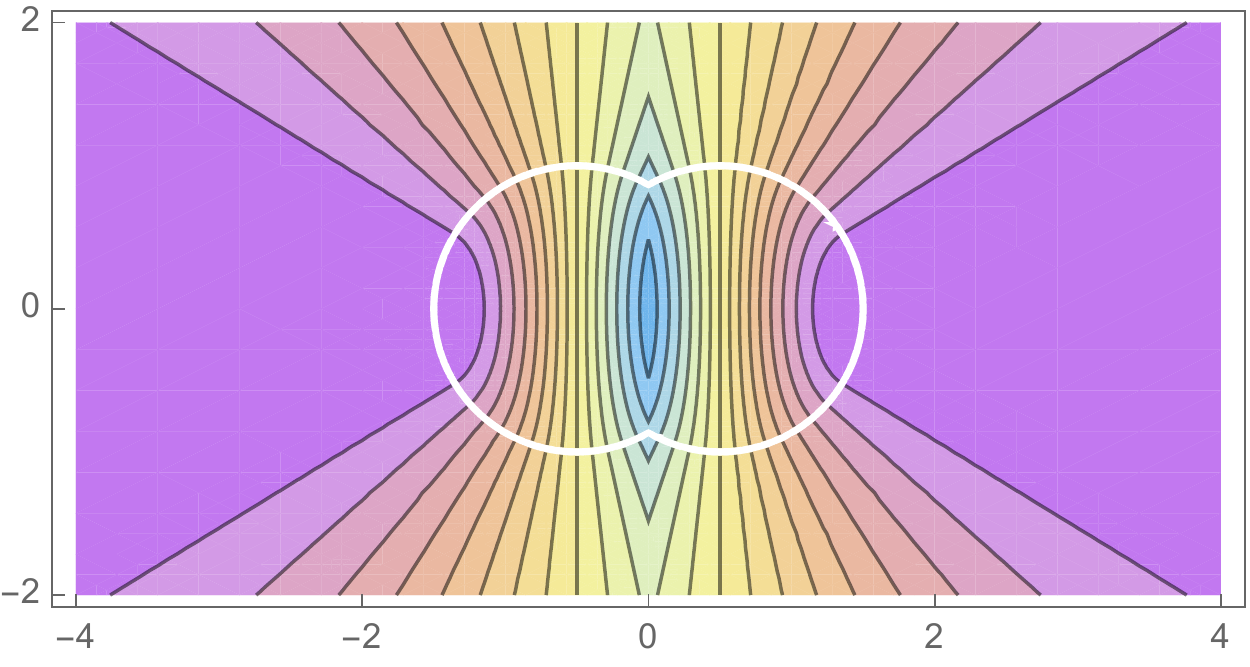} \quad \includegraphics[scale=0.55]{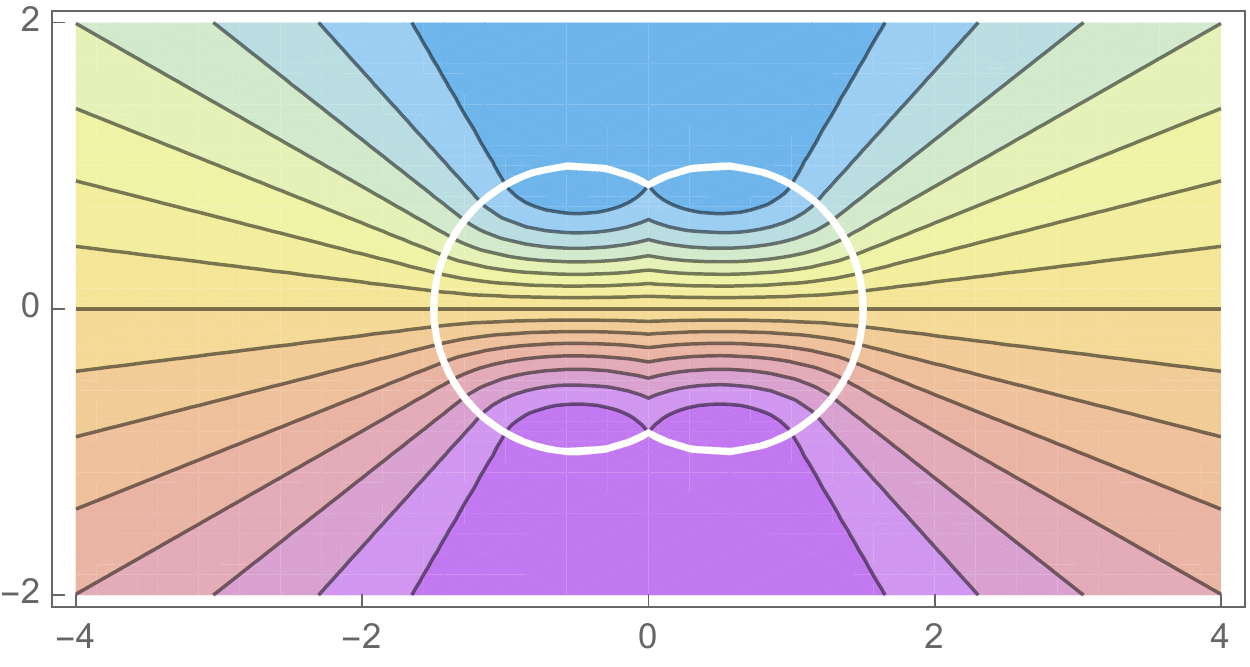} \\
\includegraphics[scale=0.55]{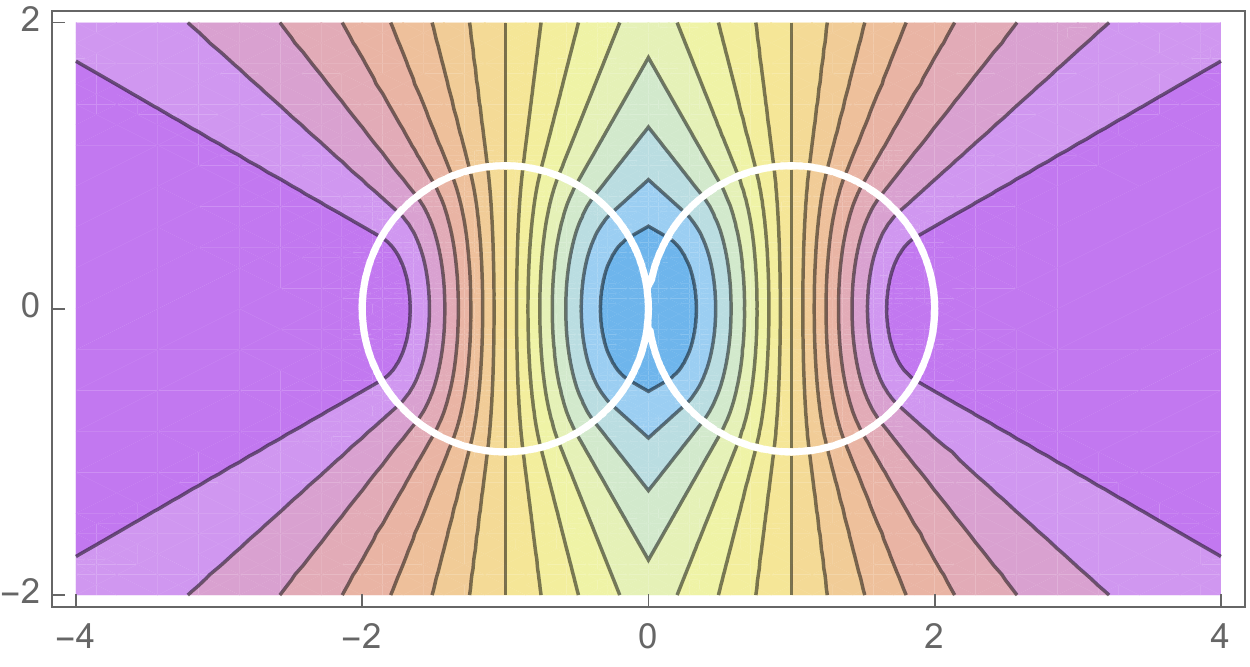} \quad \includegraphics[scale=0.55]{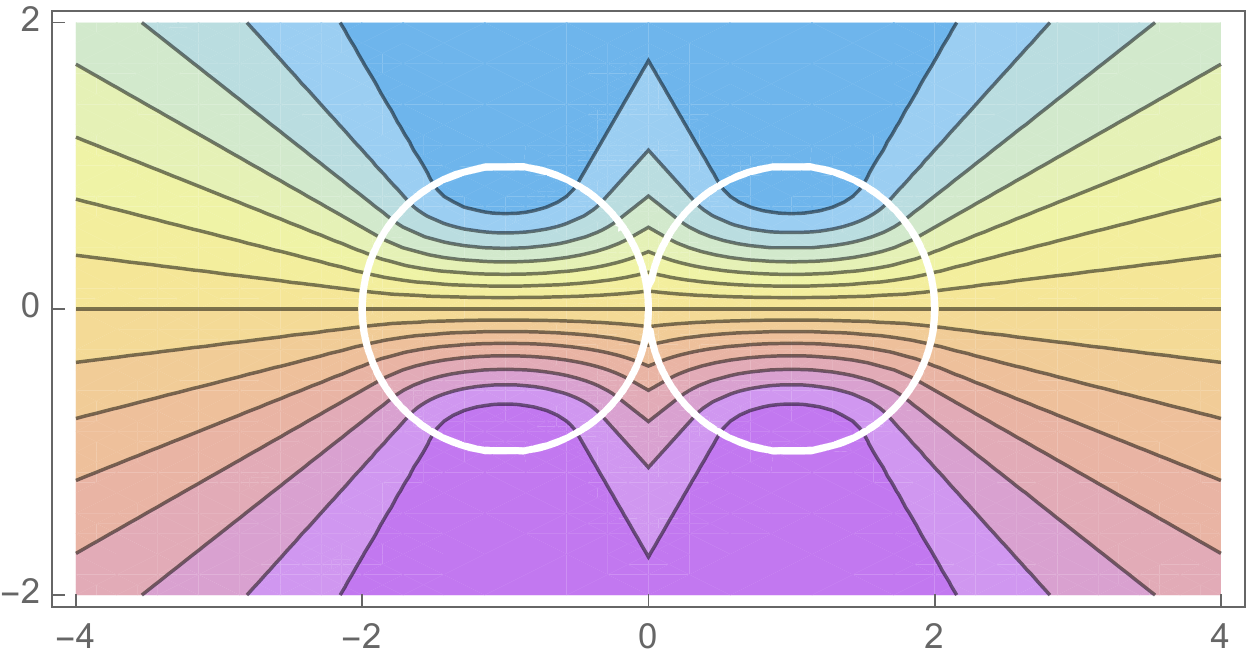} \\
\includegraphics[scale=0.55]{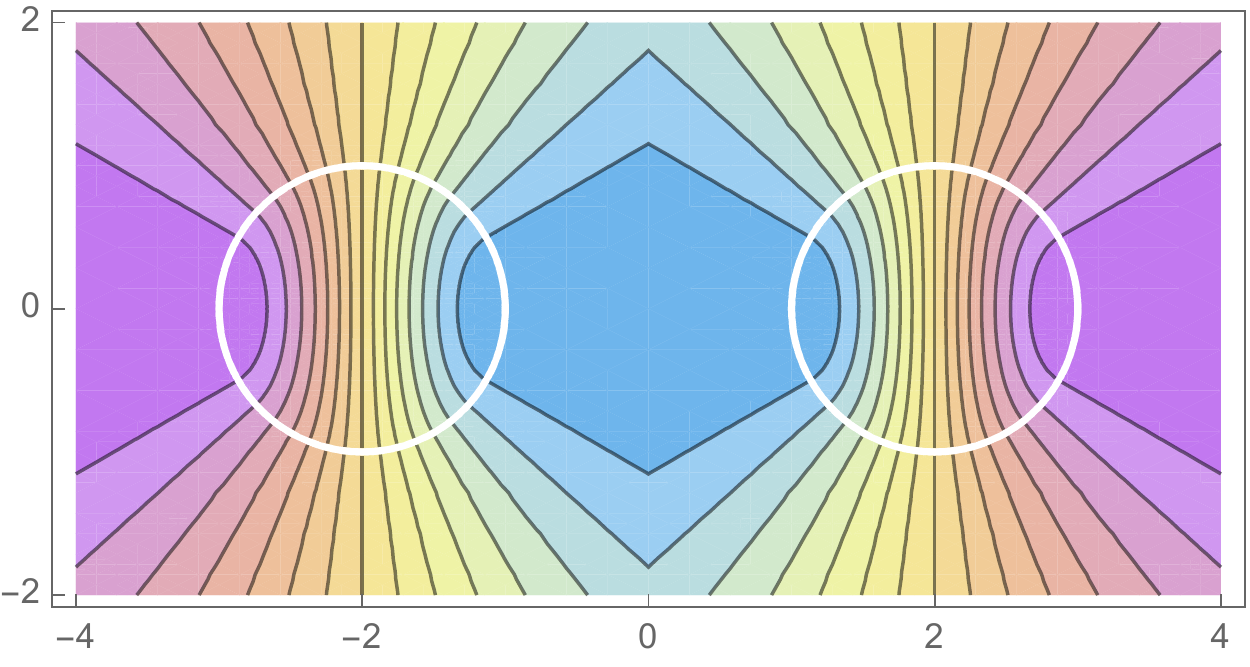} \quad \includegraphics[scale=0.55]{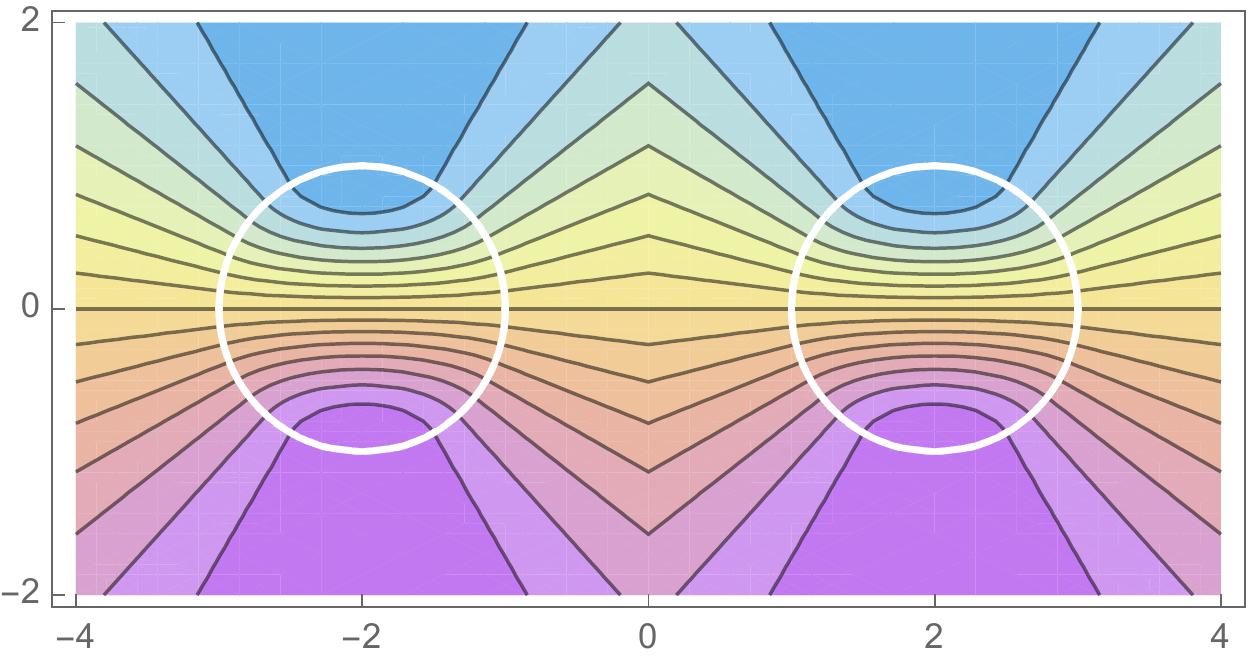}
\caption{\label{fig:trialfns} Trial function contour plots when $H$ is the left halfplane (meaning $p=(1,0),t=0$). Centering points where the trial functions vanish:
top $c=(-1/2,0)$, middle $c=(-1,0)$, bottom $c=(-2,0)$. Left side: cosine mode, from first component of the trial vector field $v(F_H(y)-c)$. Right side: sine mode,
from second component of the trial vector field. The white circles sit at radius $1$ from the point $c$ and its reflected point. All trial functions are even with respect to $H$.}
\end{center}
\end{figure}

\subsection*{Orthogonality of trial functions to the constant} We claim that there exists a unique point $c_H = c_{p,t} \in \Rn$ such that each component of the vector field
\[
v_H(y) = v(F_H(y)-c_H)
\]
is orthogonal to the constant function (the Neumann ground state) on $\Omega$, meaning
\begin{equation} \label{eq:orthog}
\int_\Omega v_H(y) \, dy  = \int_\Omega v \big( F_H(y) - c_H \big) \, dy = 0 ,
\end{equation}
and that this unique point $c_H$ depends continuously on the parameters $(p,t)$ of the halfspace $H$. These claims follow from \cite[Corollary 3]{L20a} (with $f \equiv 1$ there); in the conclusion of that result, take $c_H=-x_H$, and notice the hypotheses of the corollary are satisfied because $g$ is increasing with $g(r)>0$ when $r>0$. Incidentally, results in \cite{L20a} can also handle unbounded sets of finite volume, if desired. 

We call $c_H$ the ``center of mass point'' corresponding to the halfspace $H$ and, as in the case of the fold map defined above, will use the notation $c_{H}$ or $c_{p,t}$ depending on the context.
\begin{lemma}[Location of the center of mass] \label{le:whereiscH}
The point $c_H$ defined by condition \eqref{eq:orthog} lies in the halfspace $H$.
\end{lemma}
\begin{proof}
We prove the contrapositive. Suppose $c_H$ does not lie in $H$, so that $c_H \cdot p \geq t$. Also note if $y \notin \partial H$ then $F_H(y) \in H$ and so $F_H(y) \cdot p < t$, which implies $F_H(y) \cdot p < c_H \cdot p$. Hence the definition \eqref{eq:vdef} of $v$ gives that 
\[
v_H(y) \cdot p 
= \frac{g \big( |F_H(y) - c_H| \big)}{|F_H(y) - c_H|} \big( F_H(y) - c_H \big) \cdot p < 0 .
\]
Integrating over $y \in \Omega \setminus \partial H$ implies that \eqref{eq:orthog} does not hold, which completes the proof. 
\end{proof}
Recall $R_p$ denotes reflection in the hyperplane perpendicular to $p$ through the origin:
\[
R_p(y) = R_{p,0}(y) = y - 2(y \cdot p) p .
\]
This reflection commutes with $v$, with 
\begin{equation} \label{eq:vRp}
v \circ R_p = R_p \circ v ,
\end{equation}
because $v(R_p y)$ equals $g(|R_p y|)R_p y/|R_p y| = R_p \big( g(|y|)y/|y| \big)$, which is $R_p(v(y))$.  
\begin{lemma}[Reflection invariance of the center of mass for $t=0$] \label{le:symmetrycH}
For $p \in S^{n-1}$,  
\[
c_{-p,0} = R_p(c_{p,0}) .
\]
\end{lemma}
%
%
\begin{proof}
The hyperplane $\partial H_{p,0}$ passes through the origin, and forms the common boundary of the complementary halfspaces $H_{p,0}$ and $H_{-p,0}$. The reflection $R_p=R_{p,0}$
interchanges the halfspaces, and so their fold maps are related by
\begin{equation} \label{eq:Frelation}
F_{-p,0} = R_p \circ F_{p,0} .
\end{equation}

To prove $c_{-p,0}=R_p(c_{p,0})$, we must show that $R_p(c_{p,0})$ satisfies the condition determining the unique point $c_{-p,0}$, namely condition \eqref{eq:orthog} with $p$ and $t$
replaced by $-p$ and $0$. That is, we must show 
\[
\int_\Omega v \big( F_{-p,0}(y) - R_p(c_{p,0}) \big) \, dy = 0 .
\]
By relations \eqref{eq:vRp} and \eqref{eq:Frelation}, the left-hand side is equal to 
\begin{align*}
\int_\Omega v \big( R_p(F_{p,0}(y) - c_{p,0}) \big) \, dy
& = \int_\Omega R_p v \big( F_{p,0}(y) - c_{p,0} \big) \, dy \\
& = R_p \int_\Omega v \big( F_{p,0}(y) - c_{p,0} \big) \, dy . 
\end{align*}
The last integral indeed equals $0$, by condition \eqref{eq:orthog} for the center of mass point $c_{p,0}$.
\end{proof}

\subsection*{Orthogonality of trial functions to the first excited state}
Write $f$ for a first excited state on $\Omega$, that is, a Neumann eigenfunction of $\Omega$ corresponding to eigenvalue $\mu_2(\Omega)$. We will prove that for some halfspace $H$, the
components of the trial vector $v_H$ are orthogonal to $f$, meaning
\[
\int_\Omega v_H(y) f(y) \, dy = \int_\Omega v \big( F_H(y) - c_H \big) f(y) \, dy = 0 .
\]
Define a vector field
\[
W(p,t) = \int_\Omega v \big( F_H(y) - c_H \big) f(y) \, dy , \qquad p \in S^{n-1} , \quad t \geq 0 ,
\]
where $H=H_{p,t}$. It is easy to see $W$ is continuous, since $F_H$ and $c_H$ depend continuously on the parameters $p,t$ of the halfspace. The task is to show $W$ vanishes for some $(p,t)$,
which we do in \autoref{pr:vanishing} below. 

Let us investigate the vector field for large $t$. Choose $\tau>0$ large enough such that $\Omega \subset H_{p,\tau}$ for all $p \in S^{n-1}$, which can be done since $\Omega$ is bounded. 
The next lemma shows that when $t=\tau$, the vector field $W$ can be evaluated explicitly, and it does not depend on $p$.  
\begin{lemma}[Large positive $t$] \label{le:WT}
For all $p \in S^{n-1}$, 
\[
c_{p,\tau}=0 \qquad \text{and} \qquad W(p,\tau) = w ,
\] 
where $w = \int_\Omega v f \, dy$ is a constant vector (independent of $p$). 
\end{lemma}
\begin{proof}
Since $\Omega \subset H_{p,\tau}$, the fold map fixes $\Omega$, so that $F_{p,\tau}(y)=y$ for all $y \in \Omega$. Hence  
\[
\int_\Omega v \big( F_{p,\tau}(y) - 0 \big) \, dy
= \int_\Omega v(y) \, dy = 0
\]
by the normalization \eqref{eq:centroid}. Therefore the center of mass relation \eqref{eq:orthog} holds with $c_{p,\tau}=0$. Hence $W(p,\tau)  = \int_\Omega v(y) f(y) \, dy = w$. 
\end{proof}
\begin{proposition}[Vanishing of the vector field] \label{pr:vanishing}
$W(p,t)=0$ for some $p \in S^{n-1}$ and $t \in [0,\tau]$.
\end{proposition}
\begin{proof}
Suppose $W(p,t) \neq 0$ for all $p,t$, and define $\phi=W/|W|$, so that 
\[
\phi : S^{n-1} \times [0,\tau] \to S^{n-1} .
\]
For $t=\tau$ the map is constant, with 
\[
\phi(p,\tau) = \frac{w}{|w|} 
\]
by \autoref{le:WT}. Thus $\phi(p,\tau)$ has degree $0$ as a map from the sphere to itself. 

When $t=0$, we find 
\[
v \big( F_{-p,0}(y) - c_{-p,0} \big) = R_p v \big( F_{p,0}(y) - c_{p,0} \big) 
\]
by \eqref{eq:vRp}, \eqref{eq:Frelation} and \autoref{le:symmetrycH}. Multiplying the last equation by $f(y)$ and integrating over $\Omega$ implies $W(-p,0) = R_p(W(p,0))$, and so 
\[
\phi(-p,0) = R_p(\phi(p,0)) , \qquad p \in S^{n-1} .
\]
That is, $\phi(p,0)$ satisfies the reflection symmetry hypothesis of Petrides's \autoref{th:degree}. Hence $\phi(p,0)$ has nonzero degree, which is impossible because degree is a homotopy invariant. Therefore $W(p,t) = 0$ for some $p,t$. 
\end{proof}

The proof of \autoref{th:Euclidean2ball} will be completed in the next section. 

\subsection{Similarities and differences with the trial functions and homotopy argument of Bucur and Henrot\label{sec:simdiff}} As explained in the Introduction and at the beginning of this section, the
trial function construction we are using developed in stages through work of Szeg\H{o} \cite{S54}, Weinberger \cite{W56}, Nadirashvili \cite{N02}, Girouard, Nadirashvili and Polterovich \cite{GNP09},
Bucur and Henrot \cite{BH19}, Girouard and Laugesen \cite{GL19}, and now the current paper. 

The family of trial functions $v(F_H(y)-c)$ in this paper is the same as employed by Bucur and Henrot \cite{BH19}, although rather than parameterizing the trial functions as they did
in terms of points $(A,B) \in \Rn \times \Rn$, we parameterize in terms of $(p,t,c) \in S^{n-1} \times \R \times \Rn$, following the earlier work by Girouard, Nadirashvili and Polterovich in the disk.
The connection between the parameterizations is that $A=c$ is the ``center'' for the trial function and $B=R_{p,t}(c)$ is the reflection of that point in the hyperplane $\partial H_{p,t}$. Bucur
and Henrot restrict $A$ to lie in $H$, whereas we allow $c$ to lie anywhere in $\Rn$, but in practice this additional flexibility brings us no advantage because when orthogonality to the
constant function is imposed, one finds by \autoref{le:whereiscH} that the unique center of mass point $c_H$ must lie in $H$. 

A significant difference between their method and ours is that we reduce from a $2n$-parameter family of trial functions to an $n$-parameter family. We do so by using uniqueness of the center of
mass point (that is, orthogonality to the constant) to determine $c$ in terms of $p$ and $t$; see condition \eqref{eq:orthog}. Another difference is that they modify their $2n$-dimensional
vector field by an explicit two-step homotopy in order to reduce to a vector field whose degree they can compute explicitly. We rely instead on a reflection symmetry
observation~(\autoref{le:symmetrycH}) that allows us to invoke the degree theory result of Petrides (\autoref{th:degree}).

\section{\bf Proof of \autoref{th:Euclidean2ball}, continued --- deploying the trial functions}  \label{sec:EuclideanRayleigh}	

Fix the halfspace $H=H_{p,t}$ found in \autoref{pr:vanishing} for which the components of the vector field $v_H(y)=v(F_H(y)-c_H)$ are orthogonal to the first and second eigenfunctions of the Neumann Laplacian on $\Omega$, that is, to the constant function and $f$. In this section, we estimate the third Neumann eigenvalue by inserting the components of the vector field as trial functions into the Rayleigh quotient 
\[
Q[u] = \frac{\int_\Omega |\nabla u|^2 \, dy}{\int_\Omega u^2 \, dy} 
\]
for the Laplacian, and applying mass transplantation to arrive at a two-ball situation. These arguments go essentially as for Bucur and Henrot \cite[pp.\,344--345]{BH19}, who were in turn adapting the original techniques of Weinberger \cite{W56}. 

\subsection*{Rayleigh quotient estimate} Substituting each component $v_{H,1},\dots,v_{H,n}$ of the vector field $v_H$ into the Rayleigh characterization 
\[
\mu_3(\Omega) = \min \left\{ Q[u] : u \in H^1(\Omega), u \perp 1, u \perp f \right\}
\]
gives that 
\[
\mu_3(\Omega) \leq \frac{\int_\Omega |\nabla v_{H,j}|^2 \, dy}{\int_\Omega (v_{H,j})^2 \, dy} , \qquad j = 1,\ldots,n .
\]
Write
\[
\Omega_L = H \cap \Omega - c_H , \qquad \Omega_U = H \cap R_H (\Omega) - c_H ,
\]
so that the region decomposes as $\Omega = (\Omega_L + c_H) \cup R_H(\Omega_U + c_H) \cup (\Omega \cap \partial H)$. Here ``$L$'' and ``$U$'' label the portions of the region corresponding to the ``lower'' and ``upper'' halfspaces, except the reflection in the definition of $\Omega_U$ moves that piece to the lower halfspace. The sets $\Omega_L$ and $\Omega_U$ are not assumed to be disjoint, and indeed might coincide. The set $\Omega \cap \partial H$ has measure zero, and so can be neglected in what follows.

By evenness of $v_{H,j}$ with respect to reflection across $\partial H$, we may decompose the Rayleigh quotient as 
\[
\mu_3(\Omega) \leq \frac{\big( \int_{H \cap \Omega} + \int_{H \cap R_H (\Omega)} \big) |\nabla v_{H,j}|^2 \, dy}{\big( \int_{H \cap \Omega} + \int_{H \cap R_H (\Omega)} \big) (v_{H,j})^2 \, dy} .
\]
On $H$ one has $F_H(y)=y$ and so $v_H(y)=v(x)=g(r)x/|x|$ where $x=y-c_H$ and $r=|x|$. Making that change of variable gives
\[
\mu_3(\Omega) \leq \frac{\big( \int_{\Omega_L} + \int_{\Omega_U} \big) |\nabla v_j|^2 \, dx}{\big( \int_{\Omega_L} + \int_{\Omega_U} \big) v_j^2 \, dx} , \qquad j=1,\dots,n .
\]
Applying \autoref{le:trialfn} with the measures $w^*$ and $w_*$ being Lebesgue measure times the sum of indicators $1_{\Omega_L}+1_{\Omega_U}$ now implies   
\[
\mu_3(\Omega) \big( \int_{\Omega_L} + \int_{\Omega_U} \big) g(r)^2 \, dx \leq \big( \int_{\Omega_L} + \int_{\Omega_U} \big) \big( g^\prime(r)^2 + (n-1)r^{-2} g(r)^2 \big) \, dx .
\]

At this stage we will introduce a minor simplification to the usual method, which leads to needing monotonicity of only one function, rather than the usual two. (This simplification was
exploited more seriously by Freitas and Laugesen \cite[formula (14)]{FL18a}.) The technique is to subtract $\mu_3(\B \sqcup \B) \int g^2 \, dx$ from the left side of the inequality
and the equal value $\mu_2(\B) \int g^2 \, dx$ from the right side, to get 
\begin{equation} \label{eq:heq}
\big( \mu_3(\Omega) - \mu_3(\B \sqcup \B) \big) \big( \int_{\Omega_L} + \int_{\Omega_U} \big) g^2 \, dx 
\leq \big( \int_{\Omega_L} + \int_{\Omega_U} \big) h \, dx 
\end{equation}
where
\[
h(r) = g^\prime(r)^2 + (n-1)r^{-2} g(r)^2 - \mu_2(\B) g(r)^2 .
\]
Note $h$ is continuous, since by construction $g$ and $g^\prime$ are continuous for all $r \geq 0$, with $g(0)=0$. Further, $h$ has zero average over the unit ball, since by
summing as in the proof of \autoref{le:trialfn} one finds
\[
\int_\B h(r) \, dx = \sum_{j=1}^n \left( \int_\B |\nabla v_j(x)|^2 \, dx - \mu_2(\B) \int_\B v_j(x)^2 \, dx \right) = 0 ,
\]
using that $v_j$ was constructed to be a second Neumann eigenfunction of the unit ball. 

This function $h$ is strictly decreasing, by the following calculation due to Weinberger \cite[formula (2.15)]{W56}. For $0<r<1$ one differentiates directly to find
\[
h^\prime(r) = - 2 \frac{n-1}{r} \left( g^\prime(r) - \frac{g(r)}{r} \right)^{\! 2} -4\mu_2(\B) g(r) g^\prime(r) < 0 ,
\] 
where $g^{\prime\prime}$ was eliminated from the formula with the help of the Bessel-type equation \eqref{eq:besseltype}. For $r > 1$ the calculation is simpler, since $g$
is constant and $g^\prime=0$ on that range, so  that 
\[
h^\prime(r) = - 2 \frac{n-1}{r^3} g(1)^2 < 0 .
\] 
Hence $h$ is strictly decreasing. 

Since $\vm{\Omega} = 2\vm{\B}$, mass transplantatiion as in \autoref{transplantation} with $\omega \equiv 1$ implies that the right side of \eqref{eq:heq} is less than or equal to $2 \int_\B h(r) \, dx = 0$.
From the left side of \eqref{eq:heq} we conclude $\mu_3(\Omega) - \mu_3(\B \sqcup \B) \leq 0$. Equality obviously holds if $\Omega$ is a union of two disjoint balls of equal volume.
Finally, if equality holds then the equality statement of \autoref{transplantation} implies that $\Omega_L$ and $\Omega_U$ equal $\B$ up to sets of measure zero, and so $\Omega \cap H$
and $\Omega \cap H^c$ are each translates of $\B$, up to sets of measure zero. The Lipschitz boundary assumption on the open set $\Omega$ then forces $\Omega \cap H$ and
$\Omega \cap H^c$ to actually equal those translates of $\B$, and hence $\Omega$ is the union of two disjoint balls. This finishes the proof of \autoref{th:Euclidean2ball}.

\section{\bf Proof of \autoref{th:hyperbolic2ball} (hyperbolic) --- constructing the trial functions}  \label{sec:hyperbolic2ballproof}	

The Euclidean proof in the preceding sections adapts robustly to the hyperbolic setting. First we prepare the eigenvalue problem for the hyperbolic Laplacian, using the Poincar\'{e} ball model.

\subsection*{Construction of the hyperbolic Laplacian}

The Laplacian or Laplace--Beltrami operator on a Riemannian manifold with metric $\mt$ is
\[
\frac{1}{\sqrt{\det \mt}} \sum_{i,j=1}^n \partial_i (\sqrt{\det \mt} \, \mt^{ij} \, \partial_j u) .
\]
In the Poincar\'{e} ball model the metric is given by $(1-|x|^2)^{-2}$ times the Euclidean metric, that is, $\mt_{ij}(x)=(1-|x|^2)^{-2} \, \delta_{ij}$, with inverse $\mt^{ij}(x)=(1-|x|^2)^2 \, \delta_{ij}$. Hence $\sqrt{\det \mt} = (1-|x|^2)^{-n}$ and the hyperbolic Laplacian becomes 
\begin{equation} \label{eq:hypLaplacian}
\dlth u = (1-|x|^2)^n \, \nabla \cdot \big( (1-|x|^2)^{2-n} \, \nabla u \big) .
\end{equation}
\emph{Note.} Our choice of metric gives a hyperbolic space with sectional curvature $\kappa=-4$. No generality is lost by this choice, because to obtain other negative values of the curvature one simply multiplies the metric by a constant factor, which results in the Laplacian and its eigenvalues also getting multiplied by a constant. 

Consider an open set $\Omega \Subset \B$ that is compactly contained in the unit ball and has Lipschitz boundary. The Neumann eigenvalue problem for the hyperbolic Laplacian is
\begin{equation}\label{neumannproblemhyp}
\begin{split}
- \dlth u & = \eta u \ \quad \text{in $\Omega$,} \\
\frac{\partial u}{\partial\nu} & = 0 \qquad \text{on $\partial \Omega$.} 
\end{split}
\end{equation}
The associated Rayleigh quotient is 
\begin{equation*} \label{eq:rayhypQ}
Q_{h}[u] = \frac{\int_\Omega |\nabla u(x)|^2 (1-|x|^2)^2 \, d\gamma(x)}{\int_\Omega u(x)^2 \, d\gamma(x)} \,
\end{equation*}
where the hyperbolic volume element is 
\[
d\gamma(x) = \frac{1}{(1-|x|^2)^n} \, dx .
\]
The weight function $(1-|x|^2)^{-n}$ is positive and bounded on $\overline{\Omega}$, because $\Omega$ is assumed to have compact closure in the unit ball, and so the Rayleigh quotient is well defined for $u$ in the unweighted Sobolev space $H^1(\Omega)$. Since $\partial \Omega$ is assumed to be Lipschitz, $H^1(\Omega)$ imbeds compactly into $L^2(\Omega)$. Hence the Neumann spectrum of the hyperbolic Laplacian is well defined and discrete, with eigenvalues 
\[
0 = \eta_1 \leq \eta_2 \leq \eta_3 \leq \cdots \to \infty 
\] 
that are characterized by the usual minimax variational principle in terms of the Rayleigh quotient. The first eigenvalue is $\eta_1 = 0$, with constant eigenfunction, and the eigenfunctions satisfy the natural boundary condition $\partial u/\partial\nu = 0$. Thus we have arrived at the eigenvalue problem \eqref{neumannproblemhyp} studied in \autoref{th:hyperbolic2ball}. 

The eigenvalues are invariant under hyperbolic isometries applied to $\Omega$. 

\subsection*{Eigenfunctions of a ball} Consider the hyperbolic eigenvalue problem \eqref{neumannproblemhyp} on a ball $B=B(a)$ of radius $a<1$ centered at the origin. In spherical coordinates
$(r,\theta) \in [0,1) \times S^{n-1}$ we may separate variables in the form 
\[
u(r,\theta)=g(r)\tsp(\theta)
\] 
and substitute into the eigenfunction equation 
\[
-\dlth u = \eta u .
\]
By expressing the hyperbolic Laplacian \eqref{eq:hypLaplacian} in spherical coordinates, we obtain that the angular part $\tsp$ satisfies
\[
 \dlts \tsp(\theta) + \ell(\ell+n-2) \tsp(\theta) = 0
\]
where $\ell \geq 0$ is an integer and $\dlts$ denotes the spherical Laplacian. When $\ell=0$,  giving a constant function $\tsp$, the 
eigenfunctions on the ball are purely radial. For positive integers $\ell$, the  angular function $\tsp$ is a spherical harmonic and the eigenvalues have multiplicity greater than $1$.

The separation of variables shows that the radial part $g$ satisfies  
%
\begin{equation}\label{besseltypeeq}
- g^{\prime\prime}(r) = \left( \frac{n-1}{r} + \frac{2(n-2)r}{1-r^2} \right) g^\prime(r) + \left( \frac{\eta}{(1-r^2)^2} - \frac{\ell(\ell+n-2)}{r^2} \right) g(r) 
\end{equation}
for $0<r<a$, with Neumann boundary condition 
\[
g^\prime(a) = 0 .
\]
The key facts about the second eigenvalue and its eigenfunctions are given in \autoref{basic2hyperbolic}, which states that:
\begin{quotation}
the second eigenfunctions of the ball $B(a)$ have angular dependence of the form $g(r)x_j/r$ for $j=1,\dots,n$, and the eigenvalue has multiplicity $n$. The radial part $g$ has $g(0)=0$ and $g^\prime(r)>0$ for $r \in (0,a)$, and satisfies the boundary condition $g^\prime(a)=0$.
\end{quotation}
For example, in $2$ dimensions, the eigenfunctions have the form $g(r) \cos \theta$ and $g(r) \sin \theta$, and, in general, the angular parts $x_1/r,
\dots,x_n/r$ come from the spherical harmonics with $\ell=1$.

\subsection*{Construction of trial functions} Write $B$ for a ball centered at the origin whose hyperbolic volume equals half the hyperbolic volume of $\Omega$. Denote by $a \in (0,1)$ the Euclidean radius of $B$. Let $g(r)$ be the radial part of the second Neumann eigenfunction on $B$, as above, and extend it to the exterior of the ball by letting
\[
g(r) = g(a) , \qquad a < r < 1 .
\]
Notice $g$ is continuous and increasing for $r \in [0,1)$, and positive for $r > 0$, and $g^\prime(r)$ is continuous for $r \in [0,1)$. 

Define $v : \B \to \Rn$ to be the vector field
\begin{equation*} \label{eq:vdefhyp}
v(y) = g(|y|) \frac{y}{|y|} , \qquad y \in \B \setminus \{ 0 \} ,
\end{equation*}
with $v(0)=0$. This $v$ is continuous everywhere, including at the origin since $g(0)=0$. Each component $v_j(y)=g(r)y_j/r$ of the vector field is a hyperbolic Neumann eigenfunction on the
ball $B$ with eigenvalue $\eta_2(B)$. 

\subsubsection*{M\"{o}bius isometries of the ball} The trial functions will depend on a family of M\"{o}bius transformations 
\[
T_x : \overline{\B} \to \overline{\B}
\]
that are parameterized by $x \in \B$ and have the following properties: $T_0(y)=y$ is the identity, and when $x \neq 0$ the map $T_x(\cdot)$ is a M\"{o}bius self-map of the ball such that
$T_x(0)=x$ and $T_x$ fixes the points $\pm x/|x|$ on the unit sphere. In $2$ dimensions the maps can be written in complex notation as 
\[
T_x(y) = \frac{x+y}{1+\overline{x}y} , \qquad x \in \D , \ y \in \overline{\D} , 
\]
where $\D \simeq \B^2$ is the unit disk in the complex plane. In all dimensions \cite[eq.\,(26)]{A81}:
\begin{equation*}
T_x(y) = \frac{(1+2x \cdot y+|y|^2)x +(1-|x|^2)y}{1+2x \cdot y + |x|^2|y|^2} , \qquad x \in \B, \ y \in \overline{\B} .\label{eq:Mobius}
\end{equation*}
Observe $T_x(y)$ is a continuous function mapping $(x,y) \in \B \times \overline{\B}$ to $T_x(y) \in \overline{\B}$, and $T_x(\cdot)$ maps $\B$ to itself and $\partial \B$ to itself, with $T_x(0)=x$ and inverse $(T_x)^{-1} = T_{-x}$. Each mapping $T_x$ is a hyperbolic isometry, and its derivative matrix at $y$ is 
\begin{equation*} \label{eq:derivTx}
(DT_x)(y) = \frac{1-|T_x(y)|^2}{1-|y|^2} \times (\text{orthogonal matrix}) ,
\end{equation*}
by \cite[Section 2.7]{A81}. Taking the determinant shows $T_x$ has Jacobian $(1-|T_x(y)|^2)^n/(1-|y|^2)^n$. 

Hence the numerator of the hyperbolic Rayleigh quotient is invariant under M\"{o}bius transformations, because a straightforward change of variable reveals that 
\[
\int_E |\nabla u(y)|^2 (1-|y|^2)^2 \, d\gamma(y)  = \int_{T_x(E)} |\nabla (u \circ T_x^{-1})(y)|^2 (1-|y|^2)^2 \, d\gamma(y) 
\]
 whenever $E$ is an open subset of the unit ball and $u$ is $C^1$-smooth on $E$. (The denominator of the Rayleigh quotient is similarly invariant.) From the differential geometry
 perspective, the reason for the invariance is that the integrand on the left can be written intrinsically in terms of the metric as $|\nabla u(y)|^2 (1-|y|^2)^2 = |\nabla_{\! \mt} u(y)|_\mt^2$, and the
 M\"{o}bius transformation is an isometry with respect to the hyperbolic metric $\mt$.

\subsubsection*{Translational centering} The M\"{o}bius transformation enables us to impose a useful normalization. By replacing $\Omega$ with its hyperbolic translation $T_x(\Omega)$,
for some $x$, we may require 
\begin{equation} \label{eq:centroidhyp}
\int_\Omega v(y) \, d\gamma(y) = 0 ,
\end{equation}
as we now explain. The existence of such a ``hyperbolic Weinberger center of mass'' was known to Chavel \cite[p.\ 80] {C80}. A detailed proof using Brouwer's fixed point theorem
appeared later in Benguria and Linde \cite[Theorem 6.1]{BL07}. For a proof by energy minimization, yielding also uniqueness and continuous dependence of the translation with respect
to $\Omega$, see Laugesen \cite[Corollary 2]{L20b}, noting that the hypotheses there are satisfied with $f \equiv 1$, since $\int_0^1 g(r)(1-r^2)^{-1} \, dr = \infty$ for our function $g$. 

\subsubsection*{Hyperbolic folding, and the trial functions} Next we develop the hyperbolic fold map. Let 
\[
H_p = \{ y \in \B :  y \cdot p < 0 \} , \qquad p \in S^{n-1} ,
\]
be the halfball with normal vector $p$. Its boundary relative to the ball is the set $\partial H_p = \{ y \in \B :  y \cdot p = 0 \}$, which is the intersection of the ball with a hyperplane passing through the origin. Define
\[
H \equiv H_{p,t} = T_{pt}(H_p) , \qquad p \in S^{n-1}, \quad t \in [0,1) ,
\]
which is the image of the halfball under the M\"{o}bius translation $T_{pt}$. (Taking $t=0$ gives $H_{p,0}=H_p$.) The boundary relative to the ball is the hyperbolic hyperplane $\partial H_{p,t} = T_{pt}(\partial H_p)$. After writing
\[
R_p(y) = y - 2(y \cdot p) p
\]
for the reflection map across $\partial H_p$, we may define the hyperbolic reflection across $\partial H = \partial H_{p,t}$ by conjugation, as
\[
R_H \equiv R_{p,t} = T_{pt} \circ R_p \circ (T_{pt})^{-1} : \overline{\B} \to \overline{\B} .
\]
This reflection is a hyperbolic isometry. Clearly $R_{p,0}=R_p$. Define the hyperbolic ``fold map'' onto $H$ by 
\[
\Phi_H(y) \equiv \Phi_{p,t}(y) = 
\begin{cases}
y & \text{if\ } y \in H, \\
R_H(y) & \text{if\ } y \in \B \setminus H ,
\end{cases}
\]
so that the fold map fixes each point in $H$ and maps each point in $\B \setminus H$ to its hyperbolic reflection across $\partial H$.

Our trial functions will be the components of the vector field
\[
v \circ  T_{-c} \circ \Phi_H ,
\]
where $c \in \B$ is fixed. To visualize these trial functions, imagine ``centering'' the vector field at the point $c$  with the help of the transformation $T_{-c}$, and then replace its values on $\B \setminus H$ by evenly extending the vector field from $H$ via hyperbolic reflection across $\partial H$. The resulting vector field is continuously differentiable on each side of the hyperplane $\partial H$, and is continuous across the hyperplane. 

Contour plots for these trial functions can be visualized similar to the Euclidean case in \autoref{fig:trialfns}, except instead of spreading over the whole plane, the
functions are crammed into the unit disk. 

\subsection*{Orthogonality of trial functions to the constant} We claim that a unique ``center of mass point'' $c_H \in \B$ exists such that each component of the vector field
\[
v_H = v \circ T_{-c_H} \circ \Phi_H
\]
is orthogonal to the constant function (the Neumann ground state) on $\Omega$, meaning
\begin{equation} \label{eq:orthoghyp}
\int_\Omega v_H(y) \, d\gamma(y)  = \int_\Omega v \big( T_{-c_H} \circ \Phi_H(y) \big) \, d\gamma(y) = 0 .
\end{equation}
(Recall $d\gamma(y)$ is the volume element with respect to the hyperbolic metric.) The existence and uniqueness of $c_H$ follow directly from Laugesen~\cite[Corollary 3]{L20b}
with $f \equiv 1$, where the hypotheses
of that result are satisfied because $\int_0^1 g(r)(1-r^2)^{-1} \, dr = \infty$. Furthermore, that work shows that $c_H$ depends continuously on the
parameters $(p,t)$ of $H_{p,t}$. Write $c_{p,t}$ for the center of mass point $c_H$ that corresponds to the hyperbolic half{\-}space $H=H_{p,t}$. 

As in the Euclidean case, reflection commutes with $v$: 
\begin{equation} \label{eq:vRphyp}
v \circ R_p = R_p \circ v .
\end{equation}
Reflection also conjugates with the M\"{o}bius transformations, in the sense that
\begin{equation} \label{eq:Mobreflect}
(T_{R_p x} \circ R_p)(y) = (R_p \circ T_x)(y) , \qquad x \in \B , \quad y \in \overline{\B} , \quad p \in S^{n-1} ,
\end{equation}
as can be verified using the definitions. To understand the last formula, remember that $T_x$ represent a hyperbolic translation by $x$, and so the analogous formula in Euclidean space simply says that $R_p x+R_p y=R_p(x+y)$, which is geometrically obvious.

Now we can show how the center of mass point behaves under reflection. 
\begin{lemma}[Reflection invariance of the center of mass when $t=0$] \label{le:symmetrycHhyp}
For $p \in S^{n-1}$,  
\[
c_{-p,0} = R_p(c_{p,0}) .
\]
\end{lemma}
%
%
\begin{proof}
The hyperplane $\partial H_{p,0}$ passes through the origin, and forms the common boundary of the complementary halfballs $H_{p,0}$ and $H_{-p,0}$. The reflection $R_p$ interchanges the
halfballs, and so their fold maps are related by
\begin{equation} \label{eq:Frelationhyp}
\Phi_{-p,0} = R_p \circ \Phi_{p,0} .
\end{equation}

To prove $c_{-p,0}=R_p(c_{p,0})$, we must show that $R_p(c_{p,0})$ satisfies the condition determining the unique point $c_{-p,0}$, namely condition \eqref{eq:orthoghyp} with $p$ and $t$
replaced by $-p$ and $0$. That is, we must show 
\[
\int_\Omega v \big( T_{-R_p(c_{p,0})} \circ \Phi_{-p,0}(y) \big) \, d\gamma(y) = 0 .
\]
The integrand on the left is 
\[
v \circ T_{-R_p(c_{p,0})} \circ \Phi_{-p,0} = R_p \circ v \circ T_{-c_{p,0}} \circ \Phi_{p,0} = R_p \circ v_{H_{p,0}}
\]
by \eqref{eq:vRphyp}, \eqref{eq:Mobreflect} and \eqref{eq:Frelationhyp} . Integrating over $\Omega$ gives zero (as desired) on the right side, thanks to condition \eqref{eq:orthoghyp} for the center of mass $c_{p,0}$. 
\end{proof}

\subsection*{Orthogonality of trial functions to the first excited state} 
Write $f$ for a first excited state on $\Omega$, that is, a hyperbolic Neumann eigenfunction of $\Omega$ corresponding to eigenvalue $\eta_2(\Omega)$. We will prove that for some hyperbolic
halfspace $H$, the components of the trial vector $v_H$ are orthogonal to $f$, meaning
\[
\int_\Omega v \big( T_{-c_H} \circ \Phi_H(y) \big) f(y) \, d\gamma(y) = 0 .
\]
Define a vector field
\[
W(p,t) = \int_\Omega v \big( T_{-c_{p,t}} \circ \Phi_{p,t}(y) \big) f(y) \, d\gamma(y) , \qquad p \in S^{n-1} , \quad t \in [0,1) .
\]
We want to show $W$ vanishes for some $(p,t)$. Note $W$ is continuous, since the center of mass point $c_{p,t}$ and the fold map $\Phi_{p,t}$ depend continuously on $p,t$. 

First we investigate the vector field for $t$ near $1$. Choose $\tau>0$ close enough to $1$ that $\Omega \subset H_{p,\tau}$ for all $p \in S^{n-1}$, which can be done since $\Omega$ is compactly
contained in $\B$. 
The next lemma shows that when $t=\tau$, the vector field $W$ can be evaluated explicitly, and it does not depend on $p$.  
\begin{lemma}[$t$ near $1$] \label{le:WThyp}
For all $p \in S^{n-1}$, 
\[
c_{p,\tau}=0 \qquad \text{and} \qquad W(p,\tau) = w ,
\] 
where $w = \int_\Omega v(y) f(y) \, d\gamma(y)$ is a constant vector (independent of $p$). 
\end{lemma}
\begin{proof}
Since $\Omega \subset H_{p,\tau}$, the fold map fixes $\Omega$, so that $\Phi_{p,\tau}(y)=y$ for all $y \in \Omega$. Hence  
\[
\int_\Omega v \big( \tau_0 \circ \Phi_{p,\tau}(y) \big) \, d\gamma(y)
= \int_\Omega v(y) \, d\gamma(y) = 0
\]
by the normalization \eqref{eq:centroidhyp}. Therefore the center of mass relation \eqref{eq:orthoghyp} holds with $c_{p,\tau}=0$. Hence
\[
W(p,\tau) = \int_\Omega v(y) f(y) \, d\gamma(y) = w .
\]
\end{proof}
It follows that $W$ vanishes at some point: 
\begin{proposition}[Vanishing of the vector field] \label{pr:vanishinghyp}
$W(p,t)=0$ for some $p \in S^{n-1}$ and $t \in [0,\tau]$.
\end{proposition}
\begin{proof}
When $t=0$, we find 
\[
v \circ T_{- c_{-p,0}} \circ \Phi_{-p,0} = R_p \circ v \circ T_{-c_{p,0}} \circ \Phi_{p,0}  
\]
by \eqref{eq:vRphyp}, \eqref{eq:Mobreflect}, \eqref{eq:Frelationhyp} and \autoref{le:symmetrycHhyp}. Multiplying the last equation by $f(y) \, d\gamma(y)$ and integrating over $\Omega$
implies $W(-p,0) = R_p(W(p,0))$, so that $W$ satisfies the reflection symmetry condition. The rest of the proof goes exactly as in the Euclidean case in \autoref{pr:vanishing}, except
using \autoref{le:WThyp} instead of \autoref{le:WT}
\end{proof}

\subsection*{Estimating the Rayleigh quotient} Fix the parameters $(p,t)$ found in \autoref{pr:vanishinghyp}, and write $H=H_{p,t}$ for the hyperbolic halfspace, so that by the proposition and the
earlier formula \eqref{eq:orthoghyp}, the vector field $v_H(y)=v\big( T_{-c_H} \circ \Phi_H(y) \big)$ is orthogonal in $L^2(\Omega;d\gamma)$ to both the constant function and the first excited
state $f$. 

Substituting each component $v_{H,1},\dots,v_{H,n}$ of the vector field as a trial function into the Rayleigh characterization of the third eigenvalue gives that 
\[
\eta_3(\Omega) \leq Q_{h}[v_{H,j}] = \frac{\int_\Omega |\nabla v_{H,j}|^2 (1-|y|^2)^2 \, d\gamma(y)}{\int_\Omega (v_{H,j})^2 \, d\gamma(y)} , \qquad j = 1,\ldots,n .
\]
By evenness of $v_{H,j}$ with respect to hyperbolic reflection across $\partial H$, and using the invariance of the numerator and denominator integrals under hyperbolic reflection, we find 
\[
\eta_3(\Omega) \leq \frac{\big( \int_{H \cap \Omega} + \int_{H \cap R_H (\Omega)} \big) |\nabla v_{H,j}|^2 (1-|y|^2)^2 \, d\gamma(y)}{\big( \int_{H \cap \Omega} + \int_{H \cap R_H (\Omega)} \big)
(v_{H,j})^2 \, d\gamma(y)} .
\]

Write
\[
\Omega_L = T_{ - c_H}(H \cap \Omega) , \qquad \Omega_U = T_{- c_H}(H \cap R_H (\Omega)) ,
\]
so that the region decomposes as $\Omega =T_{c_H}(\Omega_L) \cup R_H(T_{c_H}(\Omega_U)) \cup (\Omega \cap \partial H)$.  The analogy with the Euclidean case earlier in the paper should at this
point be clear. On $H$ one has $\Phi_H(y)=y$ and so $v_H(y)=v(x)$ where $x=T_{-c_H}(y)$. Changing variable from $y$ to $x$ with the help of M\"{o}bius invariance of the numerator and denominator integrals, we obtain the inequality
\[
\eta_3(\Omega) \leq 
\frac{\big( \int_{\Omega_L} + \int_{\Omega_U} \big) |\nabla v_j|^2 (1-|x|^2)^2 \, d\gamma(x)}{\big( \int_{\Omega_L} + \int_{\Omega_U} \big) v_j^2 \, d\gamma(x)} , \qquad j=1,\dots,n .
\]
Applying \autoref{le:trialfn} with 
\[
dw^*(x)=(1-|x|^2)^2 (1_{\Omega_L}+1_{\Omega_U})(x) \, d\gamma(x) , \qquad dw_*(x)=(1_{\Omega_L}+1_{\Omega_U})(x) \, d\gamma(x) ,
\]
we deduce that 
\[
\begin{split}
& \eta_3(\Omega) \big( \int_{\Omega_L} + \int_{\Omega_U} \big) g(r)^2 \, d\gamma(x) \\
& \leq \big( \int_{\Omega_L} + \int_{\Omega_U} \big) \big( g^\prime(r)^2 + (n-1)r^{-2} g(r)^2 \big) (1-r^2)^2 \, d\gamma(x) 
\end{split}
\]
where $r=|x|$. 

Subtracting $\eta_3(B \sqcup B) \int g^2 \, d\gamma(x)$ from the left side of the inequality and the equal value $\eta_2(B) \int g^2 \, d\gamma(x)$ from the right side gives that 
\begin{equation} \label{eq:heqhyp}
\big( \eta_3(\Omega) - \eta_3(B \sqcup B) \big) \big( \int_{\Omega_L} + \int_{\Omega_U} \big) g(r)^2 \, d\gamma(x) \leq \big( \int_{\Omega_L} + \int_{\Omega_U} \big) h(r) \, d\gamma(x) 
\end{equation}
where
\[
h(r) = \big( g^\prime(r)^2 + (n-1)r^{-2} g(r)^2 \big) (1-r^2)^2  - \eta_2(B) g(r)^2 .
\]
Note $h$ is continuous, since by construction $g$ and $g^\prime$ are continuous for $0 \leq r < 1$, with $g(0)=0$. Further, $h$ has integral zero over the ball $B$, since by \eqref{eq:threestar} one finds
\begin{align*}
\int_B h(r) \, d\gamma(x) 
& = \sum_{j=1}^n \left( \int_B |\nabla v_j(x)|^2 (1-|x|^2)^2 \, d\gamma(x) - \eta_2(B) \int_B v_j(x)^2 \, d\gamma(x) \right) \\
& = 0 ,
\end{align*}
using here that each $v_j$ is a hyperbolic eigenfunction on $B$ with eigenvalue $\eta_2(B)$.

The function $h$ is strictly decreasing for $0<r<a$, as one sees by differentiating directly to find
\begin{align*}
h^\prime(r) 
& = - 2 \frac{n-1}{r} (1-r^4) \left[ \left( g^\prime(r) - \frac{1-r^2}{1+r^2} \frac{g(r)}{r} \right)^{\! 2} + \frac{4r^2}{(1+r^2)^2} \frac{g(r)^2}{r^2} \right] - 4\eta_2(B) g(r) g^\prime(r) \\
& < 0 ,
\end{align*}
where $g^{\prime\prime}$ was eliminated from the formula using the Bessel-type equation \eqref{besseltypeeq} with $\ell=1$. (Equivalent derivative calculations were given by Ashbaugh and Benguria \cite[formula (6.1)]{AB95} and Xu \cite[p.~158]{X95}.) For $a<r<1$ the calculation is simpler, because $g$ is constant and $g^\prime=0$ on that range, and so  
\[
h^\prime(r) = (n-1) g(a)^2 \, \frac{d\ }{dr} \, r^{-2} (1-r^2)^2 < 0 .
\] 
Hence $h$ is strictly decreasing for $0<r<1$. 

Since $\Omega$ has twice the hyperbolic volume of $B$, the mass transplantation \autoref{transplantation} applied with the hyperbolic volume weight $\omega(r)=(1-r^2)^{-n}$ implies that the right side of \eqref{eq:heqhyp} is less than or equal to $2 \int_B h(r) \, d\gamma(x) = 0$. (This weight $\omega$ is defined only for $0 \leq r < 1$, but that is enough for applying \autoref{transplantation} since $\Omega_L, \Omega_U$ and $B$ all lie in the unit disk.) Equality certainly holds if $\Omega$ is a union of two disjoint balls of equal hyperbolic volume. In the other direction, if the right side of \eqref{eq:heqhyp} equals $0$ then the equality statement of \autoref{transplantation} implies that $\Omega_L$ and $\Omega_U$ equal $B$ up to sets of measure zero, and so $\Omega \cap H$ and $\Omega \cap H^c$ are each hyperbolic translates of $B$, up to sets of measure zero. The Lipschitz boundary assumption on the open set $\Omega$ then requires $\Omega \cap H$ and $\Omega \cap H^c$ to actually equal those translates of $B$, and hence $\Omega$ is the union of two disjoint balls of equal hyperbolic volume. The proof of \autoref{th:hyperbolic2ball} is complete.

\section*{Acknowledgments}
This research was supported by the Funda\c c\~{a}o para a Ci\^{e}ncia e a Tecnologia (Portugal) through project UIDB/00208/2020 (Pedro Freitas), 
and a grant from the Simons Foundation (\#429422 to Richard Laugesen). Mark Ashbaugh was generous with his time and feedback on many issues.  We are grateful to Alexandre Girouard for suggestions that improved the paper, and especially for pointing us to the degree theory result of Petrides, by which we unlocked the proof in odd dimensions.

\appendix

\section{Second Neumann eigenfunction of ball is not radial}
\label{sec:balleigen}

Separation of variables reveals the form of the eigenfunctions for the ball, but does not say whether the second eigenfunction is purely radial or has angular dependence. The next result shows it is not radial. 
\begin{proposition}[Euclidean Laplacian]\label{basic2Euclidean} The second eigenspace of the Neumann Laplacian on the ball $B(a) \subset \Rn$ has a basis of the form $\{ g(r)x_j/r : j=1,\dots,n \}$, where the radial
part $g$ has $g(0)=0$ and $g^\prime(r)>0$ for $r \in (0,a)$, and satisfies the boundary condition $g^\prime(a)=0$.
\end{proposition}
The result is well known, and was used by Weinberger \cite{W56}. It can be proved in a variety of ways that either emphasize or de-emphasize the role of special functions. For example, an approach using Bessel functions is given by Ashbaugh and Benguria \cite[p.~562]{AB93}. 

The hyperbolic Laplacian for the Poincar\'{e} ball model with curvature $-4$ was derived in \autoref{sec:hyperbolic2ballproof}, where it was found to be 
\begin{equation*}\label{eq:hyplaplacianagain}
\dlth u = (1-|x|^2)^n \, \nabla \cdot \big( (1-|x|^2)^{2-n} \, \nabla u \big) .
\end{equation*}
The second Neumann eigenfunction of this operator too has the form $g(r)x_j/r$.
\begin{proposition}[Hyperbolic Laplacian] \label{basic2hyperbolic} The second eigenspace of the hyperbolic Neumann Laplacian on the ball $B(a), 0<a<1$, has a basis of the form $\{ g(r)x_j/r : j=1,\dots,n \}$, where
the radial part $g$ has $g(0)=0$ and $g^\prime(r)>0$ for $r \in (0,a)$, and satisfies the boundary condition $g^\prime(a)=0$.
\end{proposition}
For a proof using an ODE comparison method, see Ashbaugh and Benguria \cite[Sections 3 and 6]{AB95}, which is based on arguments of Bandle \cite{B72}, \cite[pp.\,122--128]{B80} for general weighted eigenvalue problems. Bandle stated her result in $2$ dimensions, and observed that the method extends easily to higher dimensions  \cite[p.\ 153]{B80}. To connect Ashbaugh and Benguria's notation to ours, first note that the hyperbolic distance $s$ from the origin to a point at Euclidean radius $r$ satisfies $ds/dr=(1-r^2)^{-1}$, and so $s=\arctanh r$. In their paper, the hyperbolic distance $\theta$ from the origin is twice as great, due to a different normalization, and so the relation between the variables is $\theta=2s=2 \arctanh r$.  

A proof along similar lines is given by Xu \cite[Lemma~2]{X95}. Xu's expressions are given in terms of the quantity $\alpha(s) = \frac{1}{2} \sinh (2s)=r/(1-r^2)$. 

A later approach by Ashbaugh and Benguria \cite[Section 3]{AB01} develops a systematic proof in terms of raising and lowering operators for the spherical Laplacian, with subsequent hyperbolic
analogues by Benguria and Linde \cite[Section 3]{BL07} (and although some constant factors seem to be missing from their raising or lowering relations, this does not materially affect the proofs).
These papers concern the Dirichlet eigenvalues, but the insights hold also in the Neumann case.

\bibliographystyle{plain}

\end{document}